\documentclass[a4paper,leqno,10pt]{article}

\usepackage{amsmath,amssymb,amsthm,%showkeys,
verbatim,%refcheck
}

\usepackage{tikz}
\usetikzlibrary{shapes.geometric,calc,decorations.markings,math}

\tikzstyle{nodo}=[circle,draw,fill,inner sep=0pt,minimum size=%
1.5mm]
\tikzstyle{infinito}=[circle,inner sep=0pt,minimum size=0mm]

\DeclareMathOperator{\sech}{sech}
\newcommand\R{{\mathbb R}}
\newcommand\EE{{\mathcal E}_p}
\newcommand\QQ{Q_p}

\newcommand\EEsei{{\mathcal E}_6}
\newcommand\Hmu{{H_\mu^1}}
\newcommand\Z{{\mathbb Z}}
\newcommand\mup{{\mu_p}}

\newcommand\f{\frac}
\newcommand\dx{{\,dx}}

\newcommand\dy{{\,dy}}

\newcommand\muR{\mu_6^\R}
\newcommand\KR{K_6^\R}

\newcommand\G{\mathcal G}

\newcommand\be{\begin{equation}}
\newcommand\ee{\end{equation}}

\newcommand\eps{\varepsilon}

\newtheorem{theorem}{Theorem}[section]
\newtheorem{proposition}[theorem]{Proposition}
\newtheorem{lemma}[theorem]{Lemma}
\newtheorem{corollary}[theorem]{Corollary}

\theoremstyle{remark}
\newtheorem{remark}[theorem]{Remark}
\newtheorem*{remark*}{Remark}

\theoremstyle{definition}
\newtheorem{definition}[theorem]{Definition}

\date{}

\newcommand{\zz}{{\mathbb Z}}

\theoremstyle{plain}
\newtheorem{thm}{Theorem}[section]

\theoremstyle{definition}

\theoremstyle{definition}

\theoremstyle{remark}

\tikzset{every loop/.style={min distance=10mm,in=300,out=240,looseness=10}}

\tikzset{place/.style={circle,thick,draw=blue!75,fill=blue!20,minimum
size=6mm}}
\tikzset{place2/.style={circle,thick,draw=red!75,fill=red!20,minimum
size=6mm}}

\date{}

\title{Dimensional
crossover with a continuum of critical exponents for  NLS on doubly periodic metric graphs}

\author{Riccardo Adami, Simone Dovetta, Enrico Serra, Paolo Tilli \\ \ \\{\small Dipartimento di Scienze
Matematiche ``G.L. Lagrange'', Politecnico di Torino } \\ {\small
Corso Duca degli Abruzzi, 24, 10129 Torino, Italy}}

\begin{document}

\maketitle

\begin{abstract}
We investigate the existence of ground states for the focusing nonlinear Schr\"odinger equation on a prototypical doubly periodic metric graph. When the nonlinearity power is below $4$, ground states exist for every value of the mass, while, for every nonlinearity power between $4$ (included) and $6$ (excluded), a mark of $L^2$-criticality arises, as
ground states exist if and only if the mass exceeds a threshold value that depends on the power. This phenomenon can be interpreted as a continuous transition from a two-dimensional regime, for which the only critical power is $4$, to a one-dimensional behavior, in which criticality corresponds to the power $6$. We show that such a dimensional crossover is rooted in the coexistence of one-dimensional and two-dimensional Sobolev inequalities, leading to a new family of Gagliardo-Nirenberg inequalities that account for this continuum of critical exponents.
\end{abstract}

\section{Introduction}
Since the first appearance of branched
structures in the modelization of organic molecules (\cite{rs53}), through
the development of the mathematical theory of quantum graphs
(\cite{kuchment,post}), networks (or {metric graphs}) have
provided a general and flexible tool to describe dynamics in complex
structures like systems of quantum wires, Josephson junctions,
propagation of signals through waveguides, and some related technologies. Pioneering
 studies about nonlinear systems on metric graphs appeared in \cite{ali1,ali2}, but more
recently the research on such topics has grown rapidly, and  several
results have been achieved
 on propagation of solitary waves (\cite{acfn11,caudrelier,matrasulov-int}) and
on stationary states
(\cite{matrasulov2,cfn15,noja14,pelinovsky,schneider,gnutzmann}).

\noindent
In a series of recent works (\cite{ast2014,ast2015,ast2016}) we investigated
 the problem of existence of {ground states}
for the energy functional associated to the {focusing,
$L^2$-subcritical} and {critical} nonlinear Schr\"odinger (NLS) equation
\begin{equation}
\label{schrod}
i \partial_t u (t) = - u'' (t) - |u(t)|^{p-2} u(t)
\end{equation}
on {finite non-compact metric graphs}, i.e. branched structures with a finite number of vertices and edges, and at least one infinite edge (i.e. a half-line).

 \noindent
 Specifically,
by {\em ground state} on a metric graph $\G$ we mean every
global minimizer of the energy functional
\begin{equation}\label{pre-energy}
E_p (u)  = \ \f 1 2 \int_\G  |u ' |^2 \, dx - \f 1 p
\int_\G |u|^{p} \, dx,
\end{equation}
in the class of $H^1(\G)$ functions with
fixed $L^2$-norm (or mass) $\mu > 0$. The constraint is dynamically  meaningful as the mass, as well as the energy, is conserved by the NLS
flow, and
 the problem of the existence of ground states is particularly relevant in
the physics of Bose-Einstein condensates (see e.g. Section 1 in \cite{parma} and \cite{ast2014,ast2015,ast2016}).

\medskip

\noindent

In this paper we extend the analysis of the existence of ground states
to a prototypical \emph{doubly periodic} metric graph $\G$, particularly relevant in the applications,
for which the techniques developed in
in previous works (where non-compactness was due to one or more unbounded edges)
do not apply:
a two--dimensional infinite grid isometrically embedded in $\R^2$,
with  vertices on the lattice $\Z^2$ and edges of unit length (see Figure~\ref{gridpicture}).

\begin{figure} \label{gridpicture}
\begin{center}
\begin{tikzpicture}[xscale= 0.5,yscale=0.5]
\draw[step=2,thin] (0,0) grid (8,8);
\foreach \x in {0,2,...,8} \foreach \y in {0,2,...,8} \node at (\x,\y) [nodo] {};
\foreach \x in {0,2,...,8}
{\draw[dashed,thin] (\x,8.2)--(\x,9.2) (\x,-0.2)--(\x,-1.2) (-1.2,\x)--(-0.2,\x)  (8.2,\x)--(9.2,\x); }
\end{tikzpicture}
\end{center}
\caption{The grid  $\G$.}
\end{figure}

It appears, roughly speaking, that {macroscopically} the grid $\G$ has dimension two, while microscopically it is of dimension one. This peculiarity is absent in
graphs with a finite number of half-lines, where the two-dimensional
scale is lacking, as well as in other two-dimensional structures as
$\zz^2$, where edges are
missing and there is of course no microscopic one-dimensional structure \cite{weinsteinlattice}.
The presence of two scales in $\G$
results in a transition from a one-dimensional to a two-dimensional behavior, that emerges in functional inequalities and influences the existence of ground states.
We shall refer to this phenomenon as to
{\em dimensional crossover}.

Before commenting further on this point, it is convenient
to state our  main results in a precise form. We define, for $\mu >0$,
the mass-constrained set
\begin{equation}
\label{space}
\Hmu(\G) = \left\{u\in H^1(\G)\; : \; \int_\G |u|^2\dx = \mu\right\}
\end{equation}
and the corresponding
``ground-state energy level''
\begin{equation}
  \label{defEE}
\EE(\mu) = \inf_{u\in H^1_\mu(\G)} E_p(u),
\end{equation}
considered as a function
$\EE:(0,+\infty) \to \R \cup \{ -\infty \}$ of the mass $\mu$.
By a ``ground states of mass $\mu$'' we mean a function $u\in \Hmu(\G)$ such that
\[
E_p(u) = \EE(\mu).
\]
When $p\in (2,4)$, ground states exist for every prescribed mass.
 \begin{thm}[Subcritical case] \label{subcritical}
Assume  $2 < p < 4$. Then for every $\mu > 0$ there exists a ground state
of mass $\mu$, and $\EE(\mu)<0$.
 \end{thm}

The picture changes as the exponent of the nonlinearity  increases.

 \begin{thm}[Dimensional crossover] \label{crossover}
For every $p \in [4,6]$ there
exists a critical mass $\mu_p > 0$ such that
\begin{enumerate}
\item [(i)] If $p\in (4,6)$ then ground states of mass $\mu$ exist if and
only if $\mu \geq \mu_p$, and
\begin{equation}
\label{inf46}
\EE(\mu)\,\,
\begin{cases} \,\,= 0 &\text{ if } \mu \le \mup \\[0.2em]
\,\,<0 &\text{ if } \mu > \mup.
\end{cases}
\end{equation}
\item [(ii)]  If $p=4$ then ground states of mass $\mu$ exist if $\mu >
\mu_4$, whereas they do not exist if $\mu < \mu_4$. Moreover \eqref{inf46}
is valid also when $p=4$.
 \item [(iii)] If $p = 6$ then there are no ground states,
regardless of the value of $\mu$, and
\begin{equation}
\label{inf6}
\EEsei(\mu)\,\,
\begin{cases}  \,\,= 0 &\text{ if } \mu \le \mu_6 \\[0.2em]
\,\,= -\infty &\text{ if } \mu > \mu_6.\end{cases}
\end{equation}
\end{enumerate}
\end{thm}
We point out that, when $p=4$, the existence of groud states of mass $\mu=\mu_4$
is still an open problem.
For the sake of completeness, we also mention that when $p > 6$ one has
$\EE(\mu)\equiv-\infty$ for every $\mu$,
as one can easily see by a scaling argument.

In order to interpret Theorems \ref{subcritical} and \ref{crossover},
let us recall that
in $\R^d$, 
for the minimization of the NLS energy 
under a mass constraint,
there exists a {\em critical exponent} $p_d^*$ such that
\begin{enumerate}
\item if $p < p_d^*$, for every mass $\mu > 0$ the ground-state energy level is finite and negative, and is attained by a ground state;
\item if $p > p_d^*$, for every mass $\mu > 0$ the ground-state energy level equals $-\infty$.
\end{enumerate}
It is well-known (\cite{cazenave}) that  $p_d^* = \frac 4 d + 2$ for the NLS in $\R^d$, yielding $p_1^* = 6$ for $\R$ and $p_2^*=4$ for $\R^2$. Furthermore, it has been proved in  \cite{ast2015,ast2016} that for finite non-compact graphs (i.e. graphs with finitely many edges,
at least one of them being unbounded)
the critical exponent is $6$, exactly as for $\R$. Thus the exponents considered in Theorem \ref{subcritical} are subcritical both  in  dimension one and two, which reflects into the typical subcritical flavor of the result.

In fact, the main novelty of the paper emerges in Theorem \ref{crossover} and lies in the ``splitting'' of the critical exponent $p_d^*$ induced by the twofold nature (one/two dimensional) of the grid. Indeed, on the grid $\G$:
\begin{enumerate}
\item $p = 4$ is the supremum of those exponents $p$ such
that $\EE(\mu)$ is finite and negative (and attained by a ground state) for every $\mu >0$;
\item $p = 6$ is the infimum of those exponents $p$ such
that  $\EE(\mu)=-\infty$ for every $\mu > 0$.
\end{enumerate}
Besides, let us stress another remarkable aspect of the dimensional crossover. In $\R^d$, as well as on non-compact finite graphs, the critical exponent is characterized by the existence of a {\em critical mass} in the following sense: for smaller masses every function has positive energy, while for larger masses there are functions with negative energy (as already mentioned,  on a non-compact finite graph such a critical mass arises only when $p=6$).

On the contrary, on the grid $\G$ a similar notion of critical mass (the number $\mu_p$ in Theorem \ref{crossover}) arises for every $p \in [4,6]$, so that, in this respect, {\em every exponent within this range  
is, in fact, critical} (see Remark~\ref{remgn}).
Beyond this critical mass, however, the  energy is still bounded from below and a ground state exists, as if the problem had kept trace of the subcriticality
of the exponent $p<6$ at the microscopic scale.

From the point of view of functional analysis, the dimensional crossover is  due to the
simultaneous validity, for every function $u\in W^{1,1}(\G)$, of the two inequalities
\begin{equation}
  \label{duedisug1}
  \Vert u\Vert_{L^{\infty}(\G)}\leq \Vert u'\Vert_{L^1(\G)},\qquad
  \Vert u\Vert_{L^{2}(\G)}\leq  \Vert u'\Vert_{L^1(\G)}.
\end{equation}
Of these, the former is typical of dimension one, on the model
of the well known inequality
\begin{equation}
\label{sobolevR}
\|v\|_{L^\infty(\R)} \leq \frac 1 2 \|v'\|_{L^1(\R)}\qquad
\forall v\in W^{1,1}(\R),
\end{equation}
while the latter is the formal analogue of the Sobolev inequality in $\R^2$
\begin{equation*}
\|v\|_{L^2(\R^2)} \leq C\, \|\nabla v\|_{L^1(\R^2)}\qquad
\forall v\in W^{1,1}(\R^2),
\end{equation*}
and is typical of dimension two. As discussed in Section~\ref{sec:inequalities},
either inequality in \eqref{duedisug1}
entails a particular version of the Gagliardo-Nirenberg inequality in $H^1(\G)$
(\eqref{gn1} and \eqref{gn2} respectively).
By interpolation, one obtains the  \emph{critical}
Gagliardo-Nirenberg inequalities
\begin{equation}
\label{GNinterpolated}
\int_\G |u|^p \dx \ \leq \, K_p \, \left(\int_\G |u|^2\dx\right)^\frac{p-2}2 \int_\G |u'|^2 \dx ,\qquad \forall u \in H^1(\G)
\end{equation}
which, being  valid for \emph{every} exponent $p\in [4,6]$,
give rise
to a continuum of critical exponents (see also Remark~\ref{remgn}).
Indeed, using \eqref{GNinterpolated}, the NLS energy  in \eqref{pre-energy}
can be
estimated from below as
\begin{equation} \nonumber
E_p (u) \geq  \frac12 \left( 1 - \f {2 K_p} p  \mu^\frac{p- 2}2 \right)
 \int_\G |u'|^2 \, dx
\end{equation}
which shows that $E_p(u) \ge 0$ for every $u\in \Hmu(\G)$, as soon as
\[
\mu \le \left( \frac p {2 K_p} \right)^{\frac 2 {p -2}}.
\]
The number in the right-hand--side of this inequality is the critical mass $\mu_p$ of Theorem \ref{crossover}.
\medskip

Finally we would like to point out that we have chosen the grid $\G$ to illustrate our results because it is the simplest doubly periodic metric graph, on which computations and proofs are particularly transparent. It should be clear however that many other doubly periodic graphs can be treated with the methods developed in the present work. Among these, we explicitly mention the hexagonal grid, a model for {\em graphene}. 

At the core of the results stands the double periodicity of the graph, that is responsible for the occurrence of phenomena such as the dimensional crossover. To exploit the double periodicity on a concrete given graph one might of course have to alter some parts of the proofs presented in this paper to adapt them to the particular features of the graph under study. These modifications are by no means substantial, being of a technical nature. We plan to illustrate this with the detailed study of some other particular graphs, significatively relevant for the applications, in forthcoming papers.

\section{Inequalities} \label{sec:inequalities}

In this section  we establish some fundamental inequalities for functions  on the grid.

For notational purposes, it is convenient  to
describe the grid $\G$ as isometrically
embedded
in $\R^2$,
with the lattice $\Z^2$ as set of vertices,
and an edge of length one joining every pair of adjacent vertices.
In this way, it is natural
to interpret $\G$ as the union of horizontal lines $\{H_j\}$  and
vertical lines $\{V_k\}$, which cross at every vertex $(k,j)\in\Z^2$.

As on any metric graph, to deal with the energy functional \eqref{pre-energy},
the natural functional framework
is given by the standard spaces
$L^p(\G)$ and $H^1(\G)$.
With the notation for $\G$ introduced above, for the $L^p$ norms we have
\begin{equation}
\label{notationnorm}
\begin{split}
\|u\|_{L^p(\G)}^p&=\sum_{j\in\mathbb{Z}} \|u \|^p_{L^p (H_j)}
+\sum_{k \in\mathbb{Z}}\| u\|_{L^p({V_k})}^p \\
& = \sum_{j\in\mathbb{Z}}\int_{H_j}|u(x)|^pdx+\sum_{k
\in\mathbb{Z}}\int_{V_k}|u(x)|^p dx<\infty
\end{split}
\end{equation}
and
\begin{equation}\label{normaLinf}
\|u\|_{L^\infty(\G)}=
\sup_{j,k}\left\{\|u\|_{L^\infty(H_j)},\|u\|_{L^\infty(V_k)}\right\},
\end{equation}
while
\begin{equation*}
\| u \|_{H^1 (\G)}^2 \ = \ \| u \|_{L^2 (\G)}^2 + \| u
' \|_{L^2 (\G)}^2.
\end{equation*}
Here, as usual, $H^1 (\G)$ denotes the space of functions on $\G$
whose
restriction to every horizontal  and vertical line
belongs to $H^1 (\R)$, and that, in addition, are continuous at
every vertex of $\G$.
In Theorem \ref{thm:sobolev2} we shall also need the space $W^{1,1} (\G)$,
similarly defined as the space of functions on $\G$
whose
restriction to every horizontal  and vertical line
belongs to $W^{1,1}(\R)$ and that, in addition, are continuous at
every vertex. 

\begin{remark*}
In the following, symbols like $\|u\|_p$ stand for $\|u\|_{L^p(\G)}$. When the
domain of integration is different from $\G$, it will always be indicated in the norm.
\end{remark*}

First we recall the standard Gagliardo-Nirenberg inequality, which
(up to a multiplicative constant $C>1$ on the right-hand side)
is valid
on any noncompact metric graph (a proof  in the general framework can be
found in  \cite{ast2016}). Here, for the sake of completeness, we shall
give a short proof tailored to the grid $\G$ which, by the way, yields
a slightly sharper estimate.
\begin{theorem} [One-dimensional Gagliardo-Nirenberg inequality]
\label{thgn1}
For every $p\in[2,\infty)$ one has
\begin{equation}
\label{gn1}
\|u\|_{p}\leq
\|u\|_{2}^{\frac{1}2+\frac 1{p}}\|u'\|_{2}^{\frac{1}2-\frac 1 p}\qquad \forall u \in H^1(\G)
\end{equation}
and, moreover,
\begin{equation}
\label{gninfty}
\|u\|_{\infty} \le 
 \|u\|_{2}^{\frac 1 2}\|u'\|_{2}^{\frac 1 2} \qquad \forall u \in H^1(\G).
\end{equation}
\end{theorem}
\begin{proof}
  Since $\| u\|_p\leq \|u\|_\infty^{1-\frac 2 p}\|u\|_2^{\frac 2 p}$, it suffices to
  prove \eqref{gninfty}. On the other hand, given $u\in H^1(\G)$, we have
  $u^2\in W^{1,1}(H_j)$    for every horizontal line $H_j$
  of $\G$. Then, applying \eqref{sobolevR} with $v=u^2$ on $H_j$ yields
  \[
  \| u\|_{L^\infty(H_j)}^2 \leq \int_{H_j} |u(x)u'(x)|\,dx
  \leq
  \|u\|_{L^2(H_j)}\|u'\|_{L^2(H_j)}\leq
  \|u\|_{L^2(\G)}\|u'\|_{L^2(\G)}.
    \]
Since clearly this inequality remains true if we replace $H_j$ with
any vertical line $V_k$, \eqref{gninfty} follows immediately from \eqref{normaLinf}.
\end{proof}

As already mentioned,
inequalities like \eqref{gn1} and \eqref{gninfty} hold for every noncompact graph.
On the contrary the next inequality, and its consequences below, rely on the two-dimensional web structure of the grid $\G$.
\begin{theorem}[Two-dimensional Sobolev inequality]
\label{thm:sobolev2}
For every $u\in W^{1,1}(\G)$,
\begin{equation}
\label{sobolev2}
\|u\|_2\leq \frac 1 2 \|u'\|_1.
\end{equation}
\end{theorem}

\begin{figure}
\begin{center}
\begin{tikzpicture}[scale=0.6,thin]
\foreach \x in {-2,0,2,6,8,10} \draw[dashed] (\x,-1.2)--(\x,7.2);
\foreach \x in {0,6} \draw[dashed] (-3.2,\x)--(11.2,\x);
  \begin{scope}[very thick]
  \draw (-4,2)  --(4,2) node at (4,2.4)[anchor=south west]{$I_j$} --(4,4)--(12,4);
  \draw[dashed] (-6,2) node[anchor=north west]{$H_{j\phantom{+1}}$} --(-4,2) (12,4)--(14,4)
  node[anchor=south east]{$H_{j+1}$};
  \end{scope}
\draw (-4,4)--(4,4) (4,2)--(12,2);
\draw[dashed] (-6,4)  --(-4,4) (12,2)--(14,2);
\draw (4,2)--(4,-2) (4,4)--(4,8) ;
  \draw[dashed] (4,-2)--(4,-3.5) node[anchor=south west]{$V_k$} (4,8)--(4,9.5);
\end{tikzpicture}
\end{center}
\caption{The path  $P_j$ (thick in the picture).}
\label{pathPj}
\end{figure}
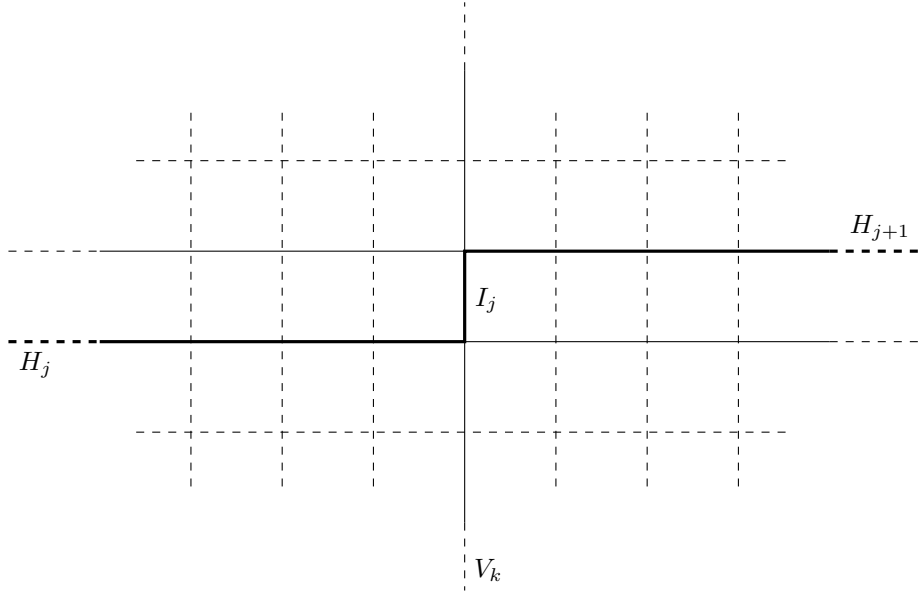

\begin{proof}
Given $u\in W^{1,1}(\G)$, we have
\begin{equation}
\label{l2}
\| u \|_2^2 = \sum_{j \in \Z} \int_{H_j} |
u (x) |^2 \, dx + \sum_{k \in \Z} \int_{V_k} | u (y) |^2 \, dy.
\end{equation}
First observe that, for each $k$, using \eqref{sobolevR} we obtain
\begin{equation}
\label{x1}
\int_{V_k} |u(y) |^2 \dy
\leq
\Vert u\Vert_{L^\infty(V_k)}\int_{V_k} |u(y) | \dy
\leq \frac 1 2
\Vert u'\Vert_{L^1(V_k)}\int_{V_k} |u(y) | \dy.
\end{equation}
Then, for each $j\in\Z$, consider the horizontal lines
$H_j$ and $H_{j+1}$, and denote by $P_j$ the path
in $\G$ obtained by joining together the halfline of $H_j$
to the left of $V_k$, the vertical segment
of $V_k$ between $H_j$ and $H_{j+1}$ (which we denote by $I_j$),
and the halfline of $H_{j+1}$ to the right of $V_k$ (see Figure~\ref{pathPj}).

Since in particular $u\in W^{1,1}(P_j)$, and the metric graph $P_j$
is isometric to $\R$,  we find from   \eqref{sobolevR}
\[
|u(y)|\leq \frac 1 2
\int_{P_j} |u'(x) | \dx\quad\forall y\in I_j
\]
and, since $I_j$ has length one, integrating this inequality over $I_j$
yields
\begin{equation}
\label{eq234}
\int_{I_j} |u(y)|\,dy\leq \frac 1 2
\int_{P_j} |u'(x) | \dx,\quad\forall j\in\Z.
\end{equation}
Now observe that
\[
V_k=\bigcup_{j\in\Z} I_j,\qquad
\bigcup_{j\in\Z} P_j=V_k\cup\bigcup_{j\in\Z} H_j,
\]
and moreover, up to a negligible set, the paths $\{P_j\}$ ($j\in\Z$)
are   mutually disjoint:
therefore, summing \eqref{eq234} over $j\in\Z$ yields
\[
\int_{V_k} |u(y)|\,dy\leq \frac 1 2
\left(
\int_{V_k} |u'(y) | \dy
+
\sum_{j}\int_{H_j} |u'(x) | \dx
\right)
=
\frac 1 2
\left(
v_k
+
\sum_{j} h_j
\right)
\]
having set, for brevity,
$v_k= \int_{V_k} |u'(y) | \dy$
and
$h_j=\int_{H_j} |u'(x) | \dx$. Combining with \eqref{x1},
and summing over $k$,
one obtains
\[
\sum\nolimits_k\int_{V_k} |u(y)|^2\,dy\leq
\frac 1 4
\sum\nolimits_k
v_k\left(
v_k
+
\sum\nolimits_{j} h_j
\right).
\]
Of course, by the symmetry of $\G$, we also have
\[
\sum\nolimits_j\int_{H_j} |u(x)|^2\,dx\leq
\frac 1 4
\sum\nolimits_j
h_j\left(
h_j
+
\sum\nolimits_{k} v_k
\right),
\]
and summing the last two inequalities we find
\[
\Vert u\Vert_{L^2(\G)}^2\leq
\frac 1 4
\left(
\sum\nolimits_k
(h_k^2+v_k^2)
+
2
\sum\nolimits_{j,k} h_j v_k
\right)
\leq
\frac 1 4
\left(
\sum\nolimits_k
h_k+v_k\right)^2
=\frac 1 4\Vert u'\Vert_{L^1(\G)}^2.
\]
\end{proof}

\begin{theorem} [Two-dimensional Gagliardo-Nirenberg inequality]
\label{thgn2}
For every $p\in[2,\infty)$ one has
\begin{equation}
\label{gn2}
\|u\|_{p}\leq C
\|u\|_{2}^{\frac 2 p}\cdot\|u'\|_{2}^{1-\frac 2 p}\qquad \forall u \in H^1(\G),
\end{equation}
where $C$ is an absolute constant.
\end{theorem}

\begin{proof}
Given $p\in [2,\infty)$,  we have
\begin{equation}\label{eqinterp}
\Vert u\Vert_p \leq \Vert u\Vert_2^{1-\theta}
\cdot\Vert u\Vert_{p+2}^{\theta}
\end{equation}
where
\begin{equation}\label{eqtheta}
\frac {1-\theta}2+\frac\theta{p+2}=\frac 1 p,\quad\text{i.e.}\quad
\theta=1-\frac 4 {p^2}.
\end{equation}
Now observe that $u\in L^\infty(\G)$ by \eqref{gninfty}, and hence
$u^{1+p/2}$ belongs to $W^{1,1}(\G)$ since $p\geq 2$.
Therefore,  we can
replace $u$
with $u^{1+p/2}$ in
\eqref{sobolev2}, thus obtaining
\[
\Vert u\Vert_{p+2}^{1 + \frac p 2}
\leq
\frac {p+2} 4 \int_\G |u(x)|^{\frac p 2}
|u'(x)|\,dx\leq
\frac {p+2} 4
\Vert u\Vert_p^{\frac p 2}
\cdot\Vert u'\Vert_2.
\]
Raising to the power $2/(p+2)$ we find
\[
\Vert u\Vert_{p+2}
\leq
C
\Vert u\Vert_p^{\frac p {p+2}}
\cdot\Vert u'\Vert_2^{\frac 2 {p+2}}, \qquad C=\sup_{p\geq 2} \left(\frac {p+2} 4\right)^{\frac 2{p+2}}
\]
(one may take e.g. $C=3/2$). Plugging this
inequality into \eqref{eqinterp} gives
\[
\Vert u\Vert_p \leq \Vert u\Vert_2^{1-\theta}
\cdot
C^\theta
\Vert u\Vert_p^{\frac {\theta p} {p+2}}
\cdot\Vert u'\Vert_2^{\frac {2\theta} {p+2}}
\]
and \eqref{gn2} follows using \eqref{eqtheta}, after elementary computations.
\end{proof}

\begin{corollary}[Interdimensional Gagliardo-Nirenberg inequality]
\label{thgn3}
There exists a universal constant $C>0$ such that,
for every $ p\in [2,\infty)$,
\begin{equation}
\label{intgn}
\|u\|_p\leq C
\|u\|_2^{1-\alpha}\|u'\|_2^{\alpha }\qquad \quad
    \forall\alpha\in \left [\frac {p-2}{2p},\frac{p-2}p\right],\qquad\forall u \in H^1(\G).
\end{equation}
In particular,
for every $p\in[4,6]$ there exists a constant $K_p$, depending only on $p$, such that
\begin{equation}
\label{gn3}
\|u\|_p^p\leq K_p
\|u\|_2^{p-2}\|u'\|_2^2\qquad \forall u \in H^1(\G).
\end{equation}
\end{corollary}

\begin{proof}
Observe that \eqref{intgn} reduces to \eqref{gn1} when $\alpha=\frac {p-2}{2p}$,
while it reduces to \eqref{gn2} when  $\alpha=\frac {p-2}{p}$. Then
\eqref{intgn} is established also for every intermediate value of $\alpha$,
since the right-hand side is a convex function of $\alpha$.

Finally,
when $p\in [4,6]$,  \eqref{gn3} is obtained letting  $\alpha=2/p$ in \eqref{intgn}
(the condition $p\in [4,6]$ guarantees that this choice of $\alpha$ is admissible).
\end{proof}

\begin{remark}\label{remgn} In $\R^d$, when dealing with the
NLS energy
\[
\frac 1 2 \Vert\nabla u\Vert_{L^2(\R^d)}^2-\frac 1 p \Vert u\Vert_{L^p(\R^d)}^p
\]
in presence of an $L^2$ mass constraint, the relevant version
of the
Gagliardo-Nirenberg (G-N) inequality
is
\begin{equation}
\label{gnRd}
\Vert u\Vert_{L^p(\R^d)}\leq
\textsc{c} \Vert u\Vert_{L^2(\R^d)}^{1-\alpha}\Vert \nabla u\Vert_{L^2(\R^d)}^\alpha ,\quad
\alpha=\frac {d(p-2)}{2p},
\end{equation}
valid as soon as $\alpha\in [0,1)$
(see \cite{libroleoni}). When $p=2+4/d$,    this inequality becomes
\emph{critical} for the NLS energy because $\alpha=2/p$ (i.e. the exponents in the inequality
become as in \eqref{gn3}),
and a critical mass $\mu_p$ comes into play.
 Now, while in \eqref{gnRd}
this \emph{critical exponent} $p=2+4/d$ is uniquely determined by the ambient space $\R^d$,
on the grid $\G$ \emph{every} $p\in [4,6]$ is critical for the NLS energy, since
one can let $\alpha=2/p$ in \eqref{intgn} (and obtain \eqref{gn3}) not just
for one particular $p$, but for every $p\in [4,6]$.

Formally, solving for $d$ in \eqref{gnRd}, for fixed $\alpha$ we can interpret \eqref{intgn}
as a G-N inequality in dimension $d=2\alpha p/(p-2)$:
we call \eqref{intgn} \emph{interdimensional} since $d$ ranges over $[1,2]$, as $\alpha$ varies
(this is in contrast with \eqref{gnRd},
where the exponent $\alpha$ is uniquely determined by $p$ and the space dimension $d$).
With this interpretation
\eqref{gn3} (which is just \eqref{intgn} with $\alpha=2/p$)
can be seen as a critical G-N inequality in dimension
$d=4/(p-2)$ so that, formally, every $p\in [4,6]$ can be seen
as the critical exponent $p=2+4/d$, in a fractal scaling dimension $d\in [1,2]$.
\end{remark}

\section{Proof of Theorem \ref{subcritical}}
\label{sec:main}
In this section we prove Theorem \ref{subcritical}.

\begin{remark}
	\label{remark inf finite}
	Note that, for every $\mu>0$ and $p<6$, the one-dimensional Ggaliardo-Nirenberg \eqref{gn1} ensures  that $\EE(\mu)$ is finite and $E_p$ is coercive on $\Hmu(\G)$(\cite{ast2016}).
\end{remark}

Recalling \eqref{space} and \eqref{defEE}, we first prove a dichotomy
lemma for
minimizing sequences, useful to prove the  existence of ground states.

\begin{lemma}[Dichotomy]
\label{propgen}
Given $\mu>0$ and $p\in(2,6)$, 
let $\{u_n\}\subset \Hmu(\G)$ be a minimizing sequence for $E_p$, i.e.
\[
\lim_{n\to\infty} E_p(u_n)=\EE(\mu),
\]
and assume that $u_n\rightharpoonup u$ weakly in $H^1(\G)$ and pointwise a.e. on $\G$.
If
\begin{equation}
\label{defm}
m:=\mu-\Vert u\Vert_2^2\in [0,\mu]
\end{equation}
denotes the loss of mass in the limit,
then either $m=0$ or $m=\mu$.
\end{lemma}

\begin{proof}
We assume that $0<m<\mu$ and seek a contradiction. According to
the Brezis--Lieb Lemma (\cite{bl}), we can write
\begin{equation}
\label{s20}
E_p(u_n) = E_p(u_n -u) + E_p(u) + o(1)\quad\text{as $n\to\infty$,}
\end{equation}
and, since $u_n\rightharpoonup u$ in $L^2(\G)$,
\begin{equation} \label{massloss}
\Vert u_n-u\Vert_2^2=
\Vert u_n\Vert_2^2+
\Vert u\Vert_2^2
-2 \langle u_n,u\rangle_2
\to \mu-\Vert u\Vert_2^2 =m
\end{equation}
as $n\to\infty$.
Now, for $n$ large enough,
\begin{equation} \nonumber \begin{split}
\EE(\mu) & \le    E_p\left(\frac{\sqrt{\mu}}{\| u_n - u \|_2}\, (u_n - u) \right) \\
 &=   \, \frac12\frac{\mu}{\| u_n - u \|_2^2}\| u_n' - u' \|_2^2 - \frac1p \frac{{\mu}^{p/2}}{\| u_n - u \|_2^p} \| u_n - u\|_p^p \\
& <   \, \frac{\mu}{\| u_n - u \|_2^2 } E_p(u_n - u),
\end{split} \end{equation}
since $ \| u_n - u\|_p \ne 0$ and $\| u_n - u \|_2^2 < \mu$. Thus,
\[
E_p(u_n - u) > \frac{\| u_n - u \|_2^2}\mu\, \EE(\mu),
\]
and by \eqref{massloss}
\[
\liminf_n E_p(u_n-u) \ge
\frac{m}\mu\, \EE(\mu).
\]
Thus, taking the liminf in \eqref{s20} we find
\begin{equation}
\label{part}
\EE(\mu) \ge \frac{m}\mu\, \EE(\mu) +E_p(u).
\end{equation}
Similarly, since  $u\not\equiv 0$ we also have
\begin{equation} \begin{split}
\EE(\mu) \le
\frac 1 2\,\frac\mu {\mu-m}
\Vert u'\Vert_2^2
-\frac 1 p
\left(\frac\mu {\mu-m}\right)^{\frac p 2}
\Vert u\Vert_p^p <  & \,
\frac\mu {\mu-m} E_p(u)
\end{split}
\end{equation}
and, as $\EE(\mu)>-\infty$ by Remark \ref{remark inf finite}, from \eqref{part} we finally obtain
\[
\EE(\mu) >  \frac{m}\mu\, \EE(\mu) + \frac{\mu-m}{\mu}\,\EE(\mu) = \EE(\mu),
\]
a contradiction. 
\end{proof}

\begin{proposition}
\label{att}
Assume that $p < 6$ and that $\EE(\mu)$ is strictly negative. Then
there exists $u\in H^1_\mu(\G)$ such that
\[
E_p(u) = \EE(\mu).
\]
\end{proposition}

\begin{proof}
Let $\{u_n\} \subset H^1_\mu(\G)$ be a minimizing sequence for $E_p$. Since $p < 6$, Remark \eqref{remark inf finite} yields
that $\EE(\mu)>-\infty$ and $u_n$ is bounded in $H^1(\G)$, and
by translating each $u_n$ (exploiting the periodicity of $\G$) we can also
assume that  $u_n$ attains its $L^\infty$-norm on a compact set ${\mathcal K}\subset\G$ independent of $n$.
Therefore, up to subsequences,  $u_n$ converges weakly in $H^1(\G)$, and strongly in $L^\infty_{\text{loc}}(\G)$, to some function $u \in H^1 (\G)$.
Setting $m := \mu - \|u\|_2^2$, from Lemma \ref{propgen} one sees that either $m=0$ or $m=\mu$. If
$m=\mu$ then $u\equiv 0$,  but in this case $u_n \to 0$ in $L^\infty(\G)$,
since in particular,
$u_n\to u\equiv 0$ uniformly on ${\mathcal K}$.
Therefore we would have
\[
E_p(u_n) \ge -\,\frac{\,1\,}p\|u_n\|_\infty^{p-2} \int_\G |u_n|^2 \dx = -\,\frac{\,\mu\,}p \|u_n\|_\infty^{p-2} \to 0,
\]
contradicting the fact that $\EE(\mu) <0$.

Thus it must be
$m=0$, so that
$u_n\to u$ strongly in $L^2(\G)$ and
therefore $u\in H^1_\mu(\G)$.
Moreover, since $u_n$ is bounded in $L^\infty(\G)$, $u_n\to u$ strongly also in $L^p(\G)$. Then
\[
E_p(u) \le \liminf_n E_p(u_n) = \EE(\mu)
\]
by weak lower semicontinuity, and the proof is complete.
\end{proof}

\begin{remark}
It is interesting to compare Proposition \ref{att} with Theorem 3.3 in \cite{ast2016}. According to that result, in a finite non-compact graph the energy threshold under which the existence of a ground state of a given mass is guaranteed equals the energy of the soliton on $\R$ with the same mass. On the contrary,
on the grid $\G$ the absence of half-lines and the periodicity pushes the energy
threshold up to zero.
This makes some proofs easier,
since finding a function with negative energy is far easier than finding a function whose energy lies below a particular negative number. In fact, this task is immediately accomplished when $p<4$, as we now show.

\end{remark}

\begin{proof}[Proof of Theorem \ref{subcritical}]
In view of Proposition \ref{att}, it suffices to construct a function in $\Hmu(\G)$ with negative energy.
Given $\mu>0$, for $\eps >0$  let
\begin{equation}
\label{nueps}
\kappa_\eps = \left(\frac{\eps\mu}2 \,\,\frac{1-e^{-2\eps}}{1+e^{-2\eps}} \right)^{1/2}
\end{equation}
and
consider the function of two variables
\[
\varphi(x,y)=\kappa_\eps  e^{-\eps (|x|+|y|)}
, \qquad (x,y)\in\R^2.
\]
Now, as described in Section~\ref{sec:inequalities}, we can consider $\G$ isometrically
embedded in $\R^2$, with its vertices
on the lattice $\Z^2$, and we
can define $u:\G\to\R$ as the restriction of $\varphi$ to the grid $\G$.
Observe that, on every horizontal line $H_j$ of $\G$, $u$ takes the form
$\kappa_\eps e^{-\eps (|x|+|j|)}$, and a similar expression holds on vertical lines.
Since for every $\lambda>0$
\[
\int_\R e^{-\lambda\eps |x|}\dx = \frac{2}{\lambda\eps}\qquad\text{and}\qquad \sum_{j\in\Z}e^{-\lambda\eps |j|} = \frac{1+e^{-\lambda\eps}}{1-e^{-\lambda\eps}},
\]
recalling \eqref{nueps} we obtain
\[
\int_\G |u_\eps|^2\dx = 2   \sum_{j\in\Z}\int_{H_j} |u_\eps|^2\dx = 2\kappa_\eps^2  \sum_{j\in\Z}e^{-2\eps |j|}\int_\R e^{-2\eps |x|}\dx  = \mu
\]
and, since $|u_\eps'(x)| = \eps |u_\eps(x)|$,
\[
\int_\G |u_\eps'|^2\dx = \eps^2 \mu.
\]
This shows in particular that $u_\eps \in \Hmu(\G)$.
Similarly,  observing that $\kappa_\eps \sim \eps\sqrt{\mu/2}$ as $\eps\to 0$,
we obtain the expansion
\[
\int_\G |u_\eps|^p\dx =
2   \sum_{j\in\Z}\int_{H_j} |u_\eps|^p\dx
=
2\,\kappa_\eps^p\, \frac{2}{\eps p}\, \frac{1+e^{-\eps p}}{1-e^{-\eps p}}
\sim C\mu^{p/2}\eps^{p-2}\quad\text{as $\eps\to 0$,}
\]
where $C$ depends  only on $p$. Therefore, as $\eps \to 0$,
\begin{equation}
\label{asymptotic}
E_p(u_\eps) \sim \frac12 \eps^2\mu - \,\frac 1 p C\mu^{p/2}\eps^{p-2},
\end{equation}
so that $E_p(u_\eps) <0$ (for $\eps$ small enough) when $p < 4$.
This proves that, when $p<4$,  $\EE(\mu)<0$ for every $\mu>0$. Moreover, since
in particular $p<6$,
Remark \eqref{remark inf finite} guarantees that $\EE(\mu)$ is finite.
The  result then follows from Proposition \ref{att}  .
\end{proof}

\section{Proof of Theorem \ref{crossover}}
\label{main2}
In the following we assume that the constants
$K_p$ in the Gagliardo-Nirenberg
inequality \eqref{gn3} are the smallest possible. In other
words,
for $p\in [4,6]$ we let
\begin{equation}
\label{Kp}
K_p =
\sup_{\overset {u\in H^1(\G)}{u\not\equiv 0}}
\QQ(u),\quad\text{where}\quad
\QQ(u)=\frac{\| u\|_p^p}{\|u\|_2^{p-2}\cdot\|u'\|_2^2}.
\end{equation}

The critical masses $\mu_p$ mentioned in Theorem~\ref{crossover}
are defined in terms of the constants $K_p$, as follows.

\begin{definition}
\label{critmass}
For every $p \in [4, 6]$ we define the  {\em critical mass $\mu_p$} as
the positive solution of
\begin{equation}
\label{critp}
\mu_p = \left(\frac{p}{2K_p}\right)^{\frac2{p-2}}.
\end{equation}
\end{definition}
This definition is natural due to the identity
\begin{equation}
\label{identity}
E_p(u)
=\frac{1}{2}\|u'\|_2^2\left(1-\frac{2}{p}\QQ(u)\mu^{\frac{p-2}2}\right)\quad
\forall u\in\Hmu(\G)
\end{equation}
which, using $\QQ(u)\leq K_p$ and \eqref{critp}, leads to the lower bound
\begin{equation}
\label{lowerb}
E_p(u)
\geq \frac{1}{2}\|u'\|_2^2\left(1-
\left({\mu}/{\mup}\right)^{\frac{p-2}2}
\right),\quad
\forall u\in\Hmu(\G),
\end{equation}
that will be widely used in the sequel.

\begin{remark}
\label{remmur}
On the real line $\R$, when $p=6$ the ground-state level
\begin{equation}
\label{gselR}
\EEsei^\R(\mu)=\inf\left\{ \frac 1 2\|w'\|_{L^2(\R)}^2-\frac 1 6 \|w\|_{L^6(\R)}^6\,\big |\,\,
w\in \Hmu(\R)\right\},\quad\mu>0
\end{equation}
is attained by a ground state if and only if $\mu=\muR$, where the
number
\begin{equation}
\label{critR}
\muR = \frac{\pi\sqrt 3}2
\end{equation}
is the critical mass of the real line (see \cite{ast2017}).
Up to sign and translations, the ground states (of mass $\muR$) are the
soliton $\varphi(x)=\sech(2x/\sqrt 3)^{1/2}$ together with all its
mass-preserving rescalings
$\varphi_\lambda(x)=\sqrt\lambda\varphi(\lambda x)$ ($\lambda>0$).
There holds
\begin{equation}
\label{defEER}
\EEsei^\R(\mu)\,\,\left\{
\begin{aligned}
&=0\ \  & \quad\text{if $\mu\leq \muR$}\\
&=-\infty  & \quad\text{if $\mu<\muR$}
\end{aligned}
\right.
\end{equation}
so that in particular ground states have zero energy. Another related quantity
is the optimal constant in the Gagliardo-Nirenber inequality on $\R$, i.e. the number
\begin{equation}
\label{miur}
\KR = \sup_{w\in H^1(\R) \atop w\not\equiv 0}
\frac{\| w\|_{L^6(\R)}^6}{\|w\|_{L^2(\R)}^4\cdot\|w'\|_{L^2(\R)}^2}=\frac 4 {\pi^2}
\end{equation}
(note that $\muR=(3/\KR)^{1/2}$, which is formally consistent with \eqref{critp}
when $p=6$).
\end{remark}
The following proposition gives  a complete picture of the problem
on the grid $\G$ when $p=6$ and, moreover, provides the exact values of $\mu_6$ and $K_6$.
\begin{proposition}
\label{murmu6}
There hold
$\mu_6 = \muR=\pi\sqrt 3/2$ and  $K_6=\KR=4/\pi^2$. Moreover
there holds $\EEsei(\mu)=\EEsei^\R(\mu)$
for every $\mu>0$, but the infimum
\begin{equation}\label{defEEsei}
\EEsei(\mu)=\inf\left\{ \frac 1 2\|u'\|_{L^2(\G)}^2-\frac 1 6 \|u\|_{L^6(\G)}^6\,\big |\,\,
u\in \Hmu(\G)\right\},\quad\mu>0
\end{equation}
is never attained.
\end{proposition}

\begin{proof} 
By a density argument,  the infimum in \eqref{gselR} 
can be restricted to functions $w\in \Hmu(\R)$ having compact support.
In fact, by a mass-preserving transformation
$w(x)\mapsto w( x/\eps^2)/\eps$, one can restrict
to functions supported in the interval $I=[-\frac 1 2,\frac 1 2]$. Then,
by intepreting this interval as one of the edges of the grid $\G$,
any function $w\in\Hmu (\R)$ supported in $I$ can be embedded
in $\Hmu(\G)$ by setting $w\equiv 0$ on $\G\setminus I$, thus providing an
admissible function in \eqref{defEEsei}. This proves that
$\EEsei(\mu)\leq \EEsei^\R(\mu)$ for every $\mu>0$. Similarly, starting from
the supremum in \eqref{miur}, by the
same argument one proves that $K_6\geq \KR$.

To prove the opposite inequalities we argue as follows.
Given a nonnegative function $u\in H^1(\G)$ ($u\not\equiv 0$), let $x_0\in\G$ be a point
where $u$ achieves   its absolute maximum $\Vert u\Vert_\infty$, and
let $P$ be any path in $\G$ such that
$x_0\in P$ and $P$ is isometric to the real line $\R$ (a natural choice for $P$ is
the horizontal/vertical line of $\G$ that contains $x_0$).
Since $u(x_0)=\Vert u\Vert_\infty$ and $u(x)\to 0$ as $x\to \pm\infty$ along $P$ (in both directions away from $x_0$),
the continuity of $u$ guarantees that $N(t)\geq 2$ for every $t\in (0,\Vert u\Vert_\infty)$, where
\begin{equation}
\label{defNt}
N(t)=\# \left\{x\in \G\,|\,\, u(x)=t\right\}
\end{equation}
counts the number of preimages in $\G$. Then, if $\widehat u\in H^1(\R)$ denotes
the symmetric rearrangement of $u$ on $\R$, applying Proposition~3.1 of~\cite{ast2015} we
obtain
\begin{equation}
\label{propsym}
\Vert (\widehat u)'\Vert_{L^2(\R)}\leq \Vert u'\Vert_{L^2(\G)},\qquad
\Vert \widehat u\Vert_{L^r(\R)}=\Vert u\Vert_{L^r(\G)}\quad\forall r
\end{equation}
so that, by the definition of $\KR$ in \eqref{miur}, we can estimate
\[
\|u\|_{L^6(\G)}^6=
\|\widehat u\|_{L^6(\R)}^6
\leq
\KR \|\widehat u\|_{L^2(\R)}^4 \|(\widehat u)'\|_{L^2(\R)}^2
\leq\KR \|  u\|_{L^2(\G)}^4 \| u'\|_{L^2(\G)}^2.
\]
Therefore, $K_6\leq \KR$ by \eqref{Kp}.  Similarly, for the NLS
energy we have
\begin{equation}
\label{ensym}
\frac 1 2\|(\widehat u)'\|_{L^2(\R)}^2-\frac 1 6 \|\widehat u\|_{L^6(\R)}^6
\leq
\frac 1 2\|u'\|_{L^2(\G)}^2-\frac 1 6 \| u\|_{L^6(\G)}^6
\end{equation}
and, since $\widehat u\in \Hmu(\R)$ whenever $u\in\Hmu(\G)$, this proves that
$\EEsei^\R(\mu)\leq \EEsei(\mu)$ for every $\mu>0$. Now assume that, for some $\mu$,
a function $u\in \Hmu(\G)$
achieves the infimum $\EEsei(\mu)$ in \eqref{defEEsei}.
Then, since  $\EEsei^\R(\mu)= \EEsei(\mu)$, \eqref{ensym}
 shows that, necessarily:
(i)\,
$\widehat u$ achieves the infimum $\EEsei^\R(\mu)$
in \eqref{gselR};
(ii)\, equality must occur in \eqref{ensym}, i.e. in \eqref{propsym}.
Now, condition (i)  entails that $\widehat u$ is a soliton
on $\R$ (necessarily of mass $\muR$),
while  (ii) implies (see Proposition~3.1 of \cite{ast2015}) that
$N(t)=2$ in \eqref{defNt}, i.e. that
$u^{-1}(t)$ has exactly $2$ elements for almost every $t\in (0,\|u\|_\infty)$:
then, since every vertex of $\G$ has degree $4$, $u$ must vanish at every vertex and is necessarily
supported in a single edge of $\G$. So $\widehat u$ has compact support too, which
is incompatible with $\widehat u$ being a soliton. This contradiction
shows the infimum in \eqref{defEEsei} is not achieved.

Finally, \eqref{critp} with $p=6$ yields $\mu_6=\sqrt{3/K_6}=\pi\sqrt{3}/2$, hence
$\mu_6=\muR$ by \eqref{critR}.
\end{proof}

\begin{proof}[Proof of Theorem \ref{crossover}]The case where $p=6$
has already been proved through Proposition~\ref{murmu6}.
The rest of the proof is divided into three parts.

\medskip

\noindent\emph{Computation of $\EE(\mu)$ when $p\in [4,6)$.}
First observe that, in the proof of Theorem \ref{subcritical},
no restriction on $p$ was used to construct $u_\eps$ and obtain \eqref{asymptotic},
which is therefore valid also when $p\geq 4$. As a consequence,
in this case, letting $\eps\to 0$
in \eqref{asymptotic} we obtain
\begin{equation} \label{zero}
\EE(\mu) 
\le 0 \qquad \forall p\geq 4,\quad\forall\mu>0.
\end{equation}
Moreover,  \eqref{lowerb} shows that $\EE(\mu)\geq 0$
when $\mu\leq\mu_p$.
This, combined with
\eqref{zero},
proves the first part of \eqref{inf46}, also when $p=4$.

Now fix a mass $\mu >\mup$ and a number $\eps>0$.
Since the quotient $\QQ(u)$ in \eqref{Kp} is unaltered if $u$ is
replaced with
$\lambda u$,
there exists $u\in \Hmu(\G)$ such that
\begin{equation}
\label{quasiopt}
\QQ(u)=\frac{\|u\|_p^p}{\mu^{\frac{p-2}2}\|u'\|_2^2}
\geq
K_p-\eps.
\end{equation}
Plugging this into \eqref{identity}, and then using \eqref{critp}, we can estimate
\[
E_p(u) \le  
\frac12 \|u'\|_2^2\left(1- \frac2p(K_p-\eps)\mu^{\frac{p-2}2}\right)
 = \frac12 \|u'\|_2^2 \left(1 - \left(\frac{\mu}{\mup}\right)^{\frac{p-2}2}+ \frac{2\eps}p \mu^{\frac{p-2}2}\right).
\]
Since $\mu>\mup$, this quantity is strictly negative if $\eps$ is small enough. Thus, for $\mu>\mup$, $\EE(\mu) <0$.
Moreover, when $p <6$, $\EE(\mu)>-\infty$ by Remark \eqref{remark inf finite}.
This proves the second part of \eqref{inf46}, also when $p=4$.

\medskip

\noindent\emph{Ground states when $p\in [4,6)$ and $\mu\not=\mu_p$.}  When $\mu >\mu_p$,
 \eqref{inf46} (valid also when $p=4$) shows that $\EE(\mu)$ is finite and negative,
 hence a ground state exists by Proposition \ref{att}.
 When $\mu <\mu_p$, $\EE(\mu) =0$ by \eqref{inf46},
  but \eqref{lowerb} reveals that
$E_p(u)>0$ for every $u\in \Hmu(\G)$. Therefore, no ground state exists
in this case.

\medskip

\noindent\emph{Ground states when $p\in (4,6)$ and $\mu=\mu_p$.}
Since by \eqref{inf46} $\EE(\mu_p)=0$, we can no longer rely
on Proposition~\ref{att}, and another argument is needed to show
 that $\EE(\mup)$ is in fact achieved.

Arguing as for \eqref{quasiopt},
let $u_n\in H^1_{\mup}(\G)$ be a
sequence of functions such that
\begin{equation}
\label{maxgn}
\lim_n Q_p(u_n)=
\lim_n \frac{\|u_n\|_p^p}{\mu_p^{\frac{p-2}2}\|u_n'\|_2^2} = K_p.
\end{equation}
We shall  bound $Q_p(u_n)$ in two different ways. First,
from the the Gagliardo-Nirenberg inequality \eqref{gn1} we obtain

\begin{equation*}
Q_p(u_n)\leq
\frac{\|u_n\|_2^{\frac{p}2+1} \|u_n'\|_2^{\frac{p}2-1}}{\mu_p^{\frac{p-2}2}\|u_n'\|_2^2}
=\frac{\mu_p^{\frac{6-p}4}}{\|u_n'\|_2^\frac{6-p}2}.
\end{equation*}
Secondly, interpolating and then using \eqref{gn3} with $p=4$, we obtain
\begin{equation*}
Q_p(u_n)\leq
\frac{\|u_n\|_\infty^{p-4}\|u_n\|_4^4}{\mu_p^{\frac{p-2}2}\|u_n'\|_2^2}
\leq \|u_n\|_\infty^{p-4} \frac{{K_4\|u_n\|_2^2\,\|u_n'\|_2^2}}{\mu_p^{\frac{p-2}2}\|u_n'\|_2^2}
=\|u_n\|_\infty^{p-4} \frac{K_4}{\mu_p^{\frac{p-4}2}}.
\end{equation*}
Recalling \eqref{maxgn}, from these two bounds
we infer that
$\|u_n'\|_2\leq C$ (compactness)
and  $\|u_n\|_\infty\geq C^{-1}$ (non-degeneracy), for some constant $C>0$ independent of $n$.
Thus $\{u_n\}$ is bounded in $H^1(\G)$ and,
 up to translations, we can also assume that each $u_n$ achieves its $L^\infty$ norm on some
  compact set  $\mathcal{K}\subset \G$ independent of $n$.
Then, up to subsequences,
$u_n\rightharpoonup u$ in $H^1(\G)$
for some $u\in H^1(\G)$,
and $u_n \to u$ in $L^\infty_{\text{loc}}(\G)$: in particular,
$u_n\to u$ uniformly on $\mathcal K$ and,
since $\Vert u_n\Vert_{L^\infty(\mathcal K)}>C^{-1}$,
$u$ is not identically zero.

Finally, writing \eqref{identity} with $u=u_n$ and $\mu=\mu_p$, since $\Vert u_n'\Vert_2\leq C$
we find
\[
\left| E_p(u_n)\right|
\leq \frac {C^2}2 \left|
1-\frac 2 p Q_p(u_n)\mu_p^{\frac{p-2}2}
\right|
=
\frac {C^2}2 \left|
1-\frac {Q_p(u_n)}{K_p}
\right|
\]
having used \eqref{critp}. Therefore,
$E_p(u_n)\to 0$  by \eqref{maxgn} and,
since $\EE(\mup) = 0$,
$u_n$ is a minimizing sequence for $E_p$, so that
Lemma \ref{propgen} applies: since we already know that
$u$ is not identically zero, we obtain that  $\|u\|_2^2 = \mup$, i.e.
$u\in H^1_{\mu_p}(\G)$.
But then $u$ is the required minimizer: indeed,
$u_n\to u$ strongly in $L^2(\G)$  hence also in $L^p(\G)$, and
by weak lower semicontinuity we obtain
\[
E_p(u) \le \liminf_n E_p(u_n) = \EE(\mup).
\]
\end{proof}

\end{document}